\documentclass[11pt]{amsart}
\usepackage{enumerate}
\usepackage{graphicx}
\usepackage{amssymb}
\usepackage{xcolor}

\newtheorem{theorem}{Theorem}[section]
\newtheorem{lemma}[theorem]{Lemma}
\newtheorem{definition}[theorem]{Definition}

\newtheorem{corollary}[theorem]{Corollary}
\newtheorem{remark}[theorem]{Remark}
\newtheorem{question}{Question}

\newcommand{\comm}[1]{}
\newcommand{\beqa}{\begin{eqnarray*}}
\newcommand{\eeqa}{\end{eqnarray*}}
\newcommand{\beqn}{\begin{eqnarray}}
\newcommand{\eeqn}{\end{eqnarray}}
\newcommand{\e}{\varepsilon}
\newcommand{\del}{\delta}
\newcommand{\Del}{\Delta}
\newcommand{\ol}{\overline}

\newcommand{\B}{\mathcal{B}}
\newcommand{\N}{\mathbb{N}}
\newcommand{\A}{\mathcal{A}}
\newcommand{\Z}{\mathcal{Z}}

\newcommand{\ds}{\displaystyle}

\newcounter{cnt1}
\newcounter{cnt2}
\newcounter{cnt3}
\newcounter{cnt4}
\newcommand{\blr}{\begin{list}{$($\roman{cnt1}$)$}
{\usecounter{cnt1} \setlength{\topsep}{0pt}
\setlength{\itemsep}{0pt}}}
\newcommand{\bla}{\begin{list}{$($\alph{cnt2}$)$}
{\usecounter{cnt2} \setlength{\topsep}{0pt}
\setlength{\itemsep}{0pt}}}
\newcommand{\bln}{\begin{list}{$($\arabic{cnt3}$)$}
{\usecounter{cnt3} \setlength{\topsep}{0pt}
\setlength{\itemsep}{0pt}}}
\newcommand{\blR}{\begin{list}{$($\Roman{cnt4}$)$}
{\usecounter{cnt4} \setlength{\topsep}{0pt}
\setlength{\itemsep}{0pt}}}
\newcommand{\el}{\end{list}}

\begin{document}

\title[(A)US, ball separation and residuality]{(Asymptotic) uniform smoothness, ball separation and residuality results}

\author[Bandyopadhyay]{Pradipta Bandyopadhyay}
\address[Pradipta Bandyopadhyay]{Stat--Math Division,
Indian Statistical Institute, 203, B.~T. Road, Kolkata
700108, India.}
\email{pradipta@isical.ac.in}

\author[Gothwal]{Deepak Gothwal}
\address[Deepak Gothwal]{Stat--Math Division, Indian Statistical Institute, 203, B.~T. Road, Kolkata
700108, India.}
\email{deepakgothwal190496@gmail.com}


\date{\today}


\begin{abstract}
In this article, we discuss a ball separation characterisation of asymptotically uniformly smooth (AUS) norms. We use this characterisation to prove the residuality of the set of equivalent AUS norms. We discuss similar residuality results for uniformly smooth norms and norms with uniform Mazur intersection property (UMIP).
\end{abstract}

\maketitle

\section{Introduction}

In geometry of Banach spaces, one often encounters properties that are isometric, and not isomorphic. That is, they depend on the specific norm and are not shared by all equivalent norms. However, among these isometric properties, there are some that satisfy the following: If a Banach space $X$ has one norm with property $P$, then ``almost every'' (residual with respect to some suitable metric) equivalent norm on $X$ also has property $P$. For example, Fabian et al. in \cite{FZZ}, proved that rotundity, locally uniform rotundity and uniform rotundity as well as having duals with these properties are such ``residual'' properties (see also \cite[Section II.4]{DGZ}).

In this note, we prove the residuality of certain geometric properties by identifying them as ball separation properties.

The study of ball separation properties in Banach spaces began with Mazur. He studied what is now called the Mazur intersection property (MIP): every closed bounded convex set in $X$ is the intersection of closed balls containing it. Among many characterisations of MIP \cite[Theorem 2.1]{GGS}, the most widely used one states: w*-denting points of $B(X^*)$ are norm dense in $S(X^*)$.

MIP is an extensively studied property with various generalisations and variations. For instance, one stronger version of MIP is the uniform Mazur intersection property (UMIP), introduced by Whitfield and Zizler in \cite{WZ} and more recently studied in \cite{BGG, Go}, which will be defined later (Definition~\ref{UMIP}). Another strengthening is obtained by considering the separation from a hyperplane. We call it the hyperplane Mazur intersection property (H-MIP) (See \cite[Definition 4.1]{Go}) which is characterised as all points of $S(X^*)$ being w*-denting points of $B(X^*)$ \cite[Theorem 4.6]{Gi}. In \cite{Go}, the second named author defined a uniform version of H-MIP (termed as H-UMIP) and showed that it characterises uniform smoothness.

Georgiev \cite{Ge} proved the residuality of MIP norms. From \cite[Corollary 5.7]{CL}, it follows that H-MIP is also residual.

In this paper, we obtain the residuality of UMIP and H-UMIP, and hence of uniform smoothness. Even though the latter can be derived from the results in \cite{FZZ}, we believe that our proof is new.

Another property we study in this paper is the asymptotic notion of uniform smoothness. Asymptotic uniform convexity and smoothness (defined in Section 2) have received a lot of attention in the literature of Banach space theory. \cite{JLPS} is a good source of information where the authors' concern is the existence of points of $\e$-Fr\'echet differentiability of Lipschitz maps. Authors in \cite{GKL} and \cite{KOS} also spend a good deal of time on the property UKK* (see the above references for definition) which is equivalent to w*-AUC. Their focus is mainly on renormings of Banach spaces with UKK* in connection with Szlenk index. Raja in \cite{Ra2} deals with existence of an AUS norm using an ordinal type of index. Other than these, \cite{DKLR, MT, MMT, OS} are some of the important works in the study of AUS.

In this note, we characterise AUS in terms of a ball separation property similar to \cite[Theorem 4.7]{Go}. We again make use of the connection between residuality and ball separation to obtain residuality of AUS norms. \cite[Remark 4.5]{DKLR} observed the residuality of AUC norms and from it, residuality of AUS norms in reflexive spaces. We prove this for general Banach spaces, not merely reflexive. As a corollary, we extend their result on existence of a norm with both AUS and AUC to general Banach spaces.

Indeed, residuality results of \cite{FZZ} are essentially obtained for convexity properties on $X$ or $X^*$ and are translated to smoothness properties on $X$ via duality, which is often not a complete duality and hence, needs reflexivity. Our results are for smoothness properties on $X$ itself via ball separation.

In the final section, we touch upon two open questions and possible approaches to answer them.

First, the residuality of Fr\'echet smoothness seems to be open. We obtain a ball separation characterisation of Fr\'echet smoothness. We hope that the techniques presented in this article would be useful in obtaining the residuality of Fr\'echet smooth norms.

Also, another long-standing open question is: ``Does UMIP imply super-reflexivity?'' Now, consider the following isometric question: If $X$ has both H-MIP and UMIP, does $X$ have H-UMIP? We observe that if the answer to the above question is affirmative, then the above open problem also has a positive answer for separable spaces.

\section{Preliminaries}
In this section, we discuss necessary definitions and terminologies.

We consider real Banach spaces only. Let $X$ be a Banach space and $X^*$ its dual. For $x\in X$ and $r>0$, we denote by $B(x, r)$ \emph{the open ball} $\{y\in X : \|x-y\|<r\}$ and by $B[x, r]$ \emph{the closed ball} $\{y\in X: \|x-y\|\leq r\}$ in $X$. We denote by $B(X)$ the \emph{closed unit ball} $\{x\in X : \|x\| \leq 1\}$ and by $S(X)$ the \emph{unit sphere} $\{x \in X : \|x\| = 1\}$.

For $x\in S(X)$, let $D(x) = \{f\in S(X^*) : f(x)=1\}$. The multivalued map $D$ is called the duality map. And we denote by $NA = \cup \{D(x) : x \in S(X)\}$, the set of all norm attaining functionals.

Consider the following quantities which find their origin in the work of Milman in \cite{Mi} under different notation.
\[
\ol{\Delta}_{X}(\e) = \inf_{x \in S(X)}\sup_{Y \subseteq X, \, dim (X/Y) <\infty}\inf_{y \in Y, \, \|y\| = \e}(\|x+y\|-1)
\] and
\[\ol{\rho}_{X}(\e) = \sup_{x \in S(X)} \inf_{Y \subseteq X, \, dim (X/Y) <\infty} \sup_{y \in Y, \, \|y\| = \e}(\|x+y\|-1)
\]
In the words of authors of \cite{JLPS}, $X$ is said to be \textit{asymptotically uniformly convex (AUC)} if for every $\e >0$, $\ol{\Delta}_{X}(\e) > 0$ and \textit{asymptotically uniformly smooth (AUS)} if $\lim_{\e \rightarrow 0}\frac{\ol{\rho}_{X}(\e)}{\e} \rightarrow 0$.
If we consider these notions in the dual space $X^*$ and the finite co-dimensional subspaces $Y$ are w*-closed, then we obtain the notions of w*-AUC and w*-AUS, respectively. We denote the modulus of w*-AUC and w*-AUS by $\ol{\Delta}_{X^*}^*(\cdot)$ and $\ol{\rho}_{X^*}^*(\cdot)$, respectively. We may skip the notation for the space in the subscript if there is no ambiguity.

For a subspace $Z$ of $X^*$, we define
\[
\|x\|_{Z} =: \sup_{f \in B(Z)} |f(x)|,
\mbox{ for all } x \in X.
\]

For a set $C \subseteq X$, a subspace $Z$ of $X^*$ and $x \in X$, let \bla
\item $d_{Z}(x, C) = \inf\{\|x-z\|_{Z} : z \in C\}$;
\item $diam_{Z}(C) = \sup\{\|x-y\|_{Z} : x, y \in C\}$;
\item $B_{Z}[x, r]:=\{y \in X : \|y-x\|_{Z} \leq r\}$ and
\item $B_{Z}(x, r):=\{y \in X : \|y-x\|_{Z} < r\}$
\el
We drop the symbol in the subscript if $Z=X^*$.

A \emph{slice} of $B(X)$ determined by $f\in X^*$ and $0<\alpha < 1$ is a set of the form
\[
S(B(X), f, \alpha):=\{y \in B(X): f(x) >\alpha\}.
\]
Observe that this notation is slightly different from the standard notation of a slice. For $x\in X$, $S(B(X^*), x, \alpha)$ will be called a \emph{w*-slice of $B(X^*)$}.

\begin{definition} \rm
We say that $x\in S(X)$ is a \emph{denting point} of $B(X)$ if
for every $\e > 0$, $x$ is contained in a slice of $B(X)$ of diameter less than $\e$.

If, in the above definition, for all $\e >0$, the same functional $f$ determines the slices of diameter less than $\e$, then $x$ is called a strongly exposed point and $f$ is called the strongly exposing functional.

A w*-denting point or a w*-strongly exposed point of $B(X^*)$ is defined analogously using w*-slices.
\end{definition}

Let $\B$ be the family of unit balls determined by the set of norms on $X$ equivalent to the original norm $\|\cdot\|$. For $B \in \B$, let us denote the corresponding norm on $X$ and $X^*$ by $\|\cdot\|_B$ and $\|\cdot\|^*_B$ respectively. Let $h$ be the Hausdorff metric on $\B$, \emph{i.e.}, for $B_1, B_2 \in \B$,
\[
h(B_1, B_2):= \inf\{\e >0: B_1 \subseteq B_2 + \e B_2 \mbox{ and } B_2 \subseteq B_1 + \e B_1\}.
\]
It is well known that $(\B, h)$ is a complete metric space. This notion of metric is motivated from the metric defined in \cite[Section 5]{CL}. This is slightly different from the metric considered for residuality results in \cite{DGZ, FZZ}. However, they are equivalent \cite{Ge}.

Any unexplained terminology or results can be found in \cite{DGZ} or \cite{Me}.

\section{AUD characterises AUC}

We begin this section with the duality between AUS and w*-AUC which follows from \cite[Proposition 2.1]{DKLR}. We will use this relationship while proving the ball separation characterisation of AUS.
\begin{theorem} \label{AUCAUS}
$X$ is AUS if and only if $X^*$ is w*-AUC.
\end{theorem}

Let us define an asymptotic notion of uniform denting introduced in \cite{Go}.
\begin{definition} \rm
We say that $X$ is asymptotically uniformly denting (I) (AUD (I)) if for every $0 <\e \leq 2$, there exists $\del(\e) >0$ such that for every $x \in S(X)$, there exists a finite co-dimensional subspace $Z$ of $X$ containing $x$ and $f \in S(X^*)$ such that
\[
diam (S(B(Z), f, f(x)(1-\del))) <\e.
\]

We say that $X^*$ is w*-AUD (I) if the finite co-dimensional subspace of $X^*$ is w*-closed and the functional defining the slice comes from $X$.
\end{definition}

Since we are going to use this notion frequently, let us denote by
\[
\Z =: \{Z \subseteq X^* : Z \mbox{ is a w*-closed finite co-dimensional subspace of } X^*\}.
\]

Following elementary lemma is an analogue of the Hahn-Banach theorem.
\begin{lemma} \label{lemext}
Let $Z \subseteq \Z$. If $u \in X$ and $\|u\|_Z=1$, then there exists $x \in S(X)$ such that $x|_Z = u|_Z$.
\end{lemma}

The following theorem relates w*-AUC and w*-AUD (I):
\begin{theorem} \label{AUDAUC}
$X^*$ is w*-AUD (I) if and only if $X^*$ is w*-AUC.
\end{theorem}

\begin{proof}
Let $X^*$ be w*-AUC and $0 <\e \leq 2$. We have, $\ol{\Del}^*(\e/8) >0$. Let $0 <\del <\ol{\Del}^*(\e/8)$.
So, for $x^* \in S(X^*)$, there exists $Y \in \Z$ such that
\[
\inf_{y^* \in Y, \, \|y^*\|=\e/8}\|x^*+y^*\|-1 >\del.
\]

\textsc{Claim}~: $\ds \inf_{y^* \in Y, \, \|y^*\| \geq \e/8} \|x^*+y^*\| = \inf_{y^* \in Y, \, \|y^*\|=\e/8}\|x^*+y^*\|$.

If not, then there exists $y_0^* \in Y$ with $\|y_0^*\| >\e/8$ such that
\[
\|x^*+y_0^*\| < \inf_{y^* \in Y, \, \|y^*\| = \e/8}\|x^*+y^*\| = A \mbox{ (say)}.
\]
Let $0 <\beta, \gamma <1$ be such that $\gamma <A-\|x^*+y_0^*\|$ and $\|\beta y_0^*\|=\e/8$.

Choose $x \in S(X)$ such that $(x^*+\beta y_0^*)(x) >\|x^*+\beta y_0^*\|-\min\{\gamma, \del/2\}$. So, we have $(x^*+\beta y_0^*)(x) >1 + \del/2$ which implies $y_0^*(x) >0$.
Also, $\|x^*+y_0^*\| < \|x^*+\beta y_0^*\|$.

We have $(x^*+y_0^*)(x) \leq \|x^*+y_0^*\| <\|x^*+\beta y_0^*\|-\gamma <(x^*+\beta y_0^*)(x)$. That is, $y_0^*(x) <\beta y_0^*(x)$. But, $0 <\beta <1$. So, $y_0^*(x) < 0$. This is a contradiction.

Therefore, $\inf_{y^* \in Y, \, \|y^*\| \geq \e/8}\|x^*+y^*\|-1 >\del$.

Clearly, $x^* \notin Y$. Let $Z = span(Y \cup \{x^*\})$. We have that, $Z$ is w*-closed. So, there exists $x \in S(X)$ such that $x^*(x) >0$ and $Y = \ker (x|_Z)$.

Now, by \cite[Lemma 3.1]{DL},
\[
diam (S(B(Z), x, x^*(x)(1-\del))) <\e.
\]
So, $X^*$ is w*-AUD (I).

Conversely, let $X^*$ be w*-AUD (I) and $0 <\e \leq 2$. Let $\del=\del(\e/5)$ obtained from the hypothesis.

Let $\beta:=\min\left\{\frac{\del}{1-\del}, \frac{\e}{20}\right\}$.

For any $x^* \in S(X^*)$, there is $Z \in \Z$ containing $x^*$ and $x \in S(X)$ such that
\[
diam (S(B(Z), x, x^*(x)(1-\del))) <\e/5.
\]
Let $Y=Z \cap \ker(x)$. Clearly, $Y \in \Z$. And $Y=\ker(x|_{Z})$.

Thus, by slight modification in the arguments of \cite[Lemma 2.8]{Go},
\[
\inf_{y^* \in Y, \, \|y^*\| = \e}\|x^*+y^*\|-1 \geq \inf_{y^* \in Y, \, \|y^*\| \geq \e}\|x^*+y^*\|-1 >\beta.
\]

So, $\ol{\Delta}^*(\e) \geq \beta >0$.
Hence, $X^*$ is w*-AUC.
\end{proof}

Analogous to the above theorem, we have the following results relating AUC and AUD (I).
\begin{theorem}
$X$ is AUC if and only if $X$ is AUD (I).
\end{theorem}

\section{Ball separation characterisation of AUS}
In order to obtain a ball separation characterisation of AUS, we need another notion of asymptotic uniform w*-denting defined below:

\begin{definition} \rm
$X^*$ is said to be w*-AUD (II) if for every $\e >0$, there is $\del >0$ such that for every $f \in S(X^*)$, there is $Z \in \Z$ containing $f$ and $x \in S(X)$ with $\|x\|_Z =1 $ such that
\[
f \in S(B(Z), x, 1-\del) \mbox{ and } diam (S(B(Z), x, 1-\del)) <\e.
\]
\end{definition}

Now, we consider the following ball separation property:
\begin{definition} \rm
$X$ is said have the asymptotic hyperplane uniform Mazur intersection property (AHUMIP) if for every $\e >0$, $M \geq 1$, there exist $\gamma, K >0$ such that for every $f \in S(X^*)$, there exists $Z \in \Z$ such that $f \in Z$ and for any closed convex set $C \subseteq X$ with $C \subseteq B_Z[0, M]$ and $\inf f(C) \geq \e$, there exist $z_0 \in X$ and $r_0 > 0$ such that $r_0 \leq K$, $C \subseteq B_Z [z_0, r_0] \leq r_0$ and $\inf f(B[z_0, r_0]) \geq \gamma$.
\end{definition}

The proof of the following lemma uses the techniques presented in \cite[Lemma 2.1]{IK}.

\begin{lemma} \label{lemsize}
Let $\e >0$, $0 <\beta <1$, $f \in S(X^*)$. Let $x \in S(X)$ and $Z \in \Z$ be such that $f \in S(B(Z), x/\|x\|_Z, \beta)$ and $diam (S(B(Z), x/\|x\|_Z, \beta)) <\e$. Then for any $1 >\gamma \geq \beta$, there exists $u \in S(X)$ with $\|u\|_Z=1$ such that $diam (S(B(Z), u, \gamma)) <2\e$ and $f \in S(B(Z), u, \gamma)$.
\end{lemma}

\begin{proof}
Let $\del = 1 - \gamma$. Put $x_1=\frac{x}{\|x\|_Z}$.
Let $z \in S(X)$ be such that $f(z) >1-\del/2$.

Define $\ds G(\alpha):= \frac{1+\alpha \beta}{\|z + \alpha x_1\|_Z}$.
We have, $G(0) \geq 1$ and $\ds \lim_{\alpha \rightarrow \infty} G(\alpha)=\beta$.

Since $1 >1-\del/2 > \beta$, there exists $\alpha_0 >0$ such that $G(\alpha_0)=1-\del/2$, \emph{i.e.},
\[
\frac{1+\alpha_0 \beta}{\|z + \alpha_0 x_1\|_Z}=1- \del/2.
\]

Let
\[
y:=\frac{z+\alpha_0 x_1}{\|z+\alpha_0 x_1\|_Z}.
\]
By Lemma ~\ref{lemext}, there exists $u \in S(X)$ such that $u|_Z = y|_Z$.

Consider $h \in S(B(Z), u, 1-\del/2)$. Then,
\beqa
1+\alpha_0 h(x_1) & \geq & h(z) + \alpha_0 h(x_1) = h(z + \alpha_0 x_1) = h(y)\|z + \alpha_0 x_1\|_Z\\
& > & (1-\del/2) \|z + \alpha_0 x_1\|_Z = 1 + \alpha_0 \beta.
\eeqa

So, $h \in S(B(Z), x_1, \beta)$. That is, $S(B(Z), u, 1-\del/2) \subseteq S(B(Z), x_1, \beta)$ and hence, $diam (S(B(Z), u, 1-\del/2)) <\e$. By applying \cite[Lemma 3.4 (b)]{Go} to $B(Z)$, $diam (S(B(Z), u, \gamma)) <2\e$. Also,
\beqa
f(u) & = & f(y) = f\left(\frac{z + \alpha_0 x_1}{\|z+\alpha_0 x_1\|_Z}\right) \geq \frac{f(z) + f(z)\alpha_0 f(x_1)}{\|z + \alpha_0 x_1\|_Z} \\
& > & f(z) \frac{1+\alpha_0 \beta}{\|z + \alpha_0 x_1\|_Z}= f(z) (1-\del/2) > (1-\del/2)^2 > 1 - \del = \gamma.
\eeqa

So, $f \in S(B(Z), u, \gamma)$.
\end{proof}

Following variants of the results from \cite{BGG} are crucial in obtaining the ball separation characterisation of AUS:

\begin{lemma} \label{lemasymhyp}
Let $X$ be a Banach space and $Z \in \Z$. If for $x, y \in X$ with $\|x\|_Z = \|y\|_Z =1$ and $\e > 0$,
\[
\{f \in B(Z) : f(x) > \e\} \subseteq \{g \in X^* : g(y) > 0\},
\]
then $\|x-y\|_Z < 2\e$.
\end{lemma}

\begin{proof}
Let $u=x|_Z$, $v=y|_Z \in S(Z^*)$. Following the arguments in the proof of \cite[Lemma 2.5]{BGG}, there exists $t > 0$ such that
\[|1-t| \leq \|u-tv\| \leq \e. \]
And
\[\|x-y\|_Z = \|u-v\| \leq 2\e. \]
\end{proof}

\begin{corollary} \label{corasymhyp}
Let $Z \in \Z$. If for $x, y \in X$ and $\e > 0$, if
\[
\emptyset \neq \{f \in B(Z) : f(x) > \e\} \subseteq \{g \in X^* : g(y) > 0\},
\]
then
\[
\left\|\frac{x}{\|x\|_Z}-\frac{y}{\|y\|_Z}\right\|_Z \leq \frac{2\e}{\|x\|_Z}.
\]
\end{corollary}

\begin{theorem} \label{thmbsep}
Let $X$ be a Banach space and $Z \in \Z$. Let $A \subseteq X$ be bounded in $\|\cdot\|_Z$ seminorm and $d_Z(0, A) >0$. If there is a non-empty w*-slice of $B(Z)$ contained in
\[
\{f \in B(X^*) : f(x) > 0 \mbox{ for all } x \in A\},
\]
then there exists a closed ball $B[x_0, r_0] \subseteq X$ such that $A \subseteq B_Z [x_0, r_0]$ and $0 \notin B[x_0, r_0]$.
\end{theorem}

\begin{proof}The proof is an adaptation of the proof of
\cite[Theorem 2.6]{BGG}.

Let $x_0 \in S(X)$ and $0 < \e < 1$ be such that
\[
\emptyset \neq \{f \in B(Z) : f(x_0) > \e\} \subseteq \{f \in X^* : f(x) > 0 \mbox{ for all } x \in A\}.
\]
Let $M = \sup \{\|x\|_Z : x \in A\}$. As in Lemma~\ref{lemasymhyp}, for all $x \in A$, there exists $t > 0$ such that $1-\e/\|x_0\|_Z \leq t \leq 1+\e/\|x_0\|_Z$ and
\[
\left\|\frac{t x}{\|x\|_Z} - \frac{x_0}{\|x_0\|_Z}\right\|_Z \leq \frac{\e}{\|x_0\|_Z}.
\]
That is,
\[
\left\|\frac{t\|x_0\|_Z}{\|x\|_Z}x - x_0\right\|_Z \leq \e.
\]
Also, since $\e < \|x_0\|_Z \leq 1$, we have $0 <1-\e/\|x_0\|_Z \leq t <2$.

Then for $\lambda \geq M/\e(1-\e/\|x_0\|_Z)$ and $0 <d < d_Z(0, A)$,
\beqa
\|x- \lambda x_0\|_Z & \leq & \left\|x- \frac{\|x\|_Z}{t\|x_0\|_Z} x_0\right\|_Z + \left|\frac{\|x\|_Z}{t\|x_0\|_Z} - \lambda\right|\|x_0\|_Z \\ & \leq & \frac{\e \|x\|_Z}{t\|x_0\|_Z} + \lambda \|x_0\|_Z - \frac{\|x\|_Z}{t}\\
& \leq & \lambda - \frac{d(1-\e/\|x_0\|_Z)}{2}.
\eeqa
So, $B\left[\lambda x_0, \lambda - \frac{d(1-\e/\|x_0\|_Z)}{2}\right]$ satisfies the desired conditions.
\end{proof}

\begin{theorem} \label{thm1}
The following are equivalent:
\bla
\item $X$ is AUS.
\item $X^*$ is w*-AUC.
\item $X^*$ is w*-AUD (I).
\item $X^*$ is w*-AUD (II).
\item $X^*$ is AHUMIP.
\el
\end{theorem}

\begin{proof}
$(a) \iff (b)$: This is Theorem ~\ref{AUCAUS}.

$(b) \iff (c)$: This is Theorem ~\ref{AUDAUC}.

$(c) \implies (d)$: Let $X^*$ has w*-AUD (I).

Let $\e >0$ be given. Then, there exists $\del >0$ such that for every $f \in S(X^*)$ there is $Z \in \Z$ containing $f$ and $x \in S(X)$ such that
\[
S(B(Z), x, f(x)(1-\del)) <\e/2.
\]
So,
\[
S(B(Z), x/\|x\|_Z, f(x/\|x\|_Z)(1-\del)) <\e/2.
\]

Now, $1-\del \geq f(x/\|x\|_Z)(1-\del)$.

By Lemma ~\ref{lemsize}, there exists $y \in S(X)$ with $\|y\|_Z = 1$ such that
\[
diam (S(B(Z), y, 1-\del)) <\e \mbox{ and } f(y) >1-\del.
\]

So, $X^*$ is w*-AUD (II).

$(d) \implies (e)$: We follow the techniques developed in \cite[Theorem 4.7]{Go}. But due to the quantitative variations, we include the proof here.

Let $\e >0$ and $M \geq 2$ be given. Let $M_1 := M + 3\e/4$.
Since $X^*$ is w*-AUD (II), with a slight variation in \cite[Lemma 3.4]{Go},
there exist $0 < \alpha < 1$ and $\del >0$ such that for every $f \in S(X^*)$, there exists $Z_1 \in \Z$ containing $f$ and $x \in S(X)$ with $\|x\|_{Z_1} = 1$ satisfying
\[
diam (S(B(Z_1), x, \alpha - \del)) <\e/4M_1 \mbox{ and } f(x) >\alpha.
\]

Choose $k \in \mathbb{N}$ such that $1/2(k-1) <\e/4M_1$. Again, by using \cite[Corollary 3.5]{Go}, there exists $0 < \gamma < 1$ such that for every $f \in S(X^*)$, there is $Z_2 \in \Z$ containing $f$ and $x \in S(X)$ with $\|x\|_{Z_2}=1$ such that
\[
diam (S(B(Z_2), x, \gamma)) <\del/10k \mbox{ and } f(x) > \frac{1 + \gamma}{2}.
\]

Then, by arguments similar to \cite[Lemma 3.4 (a)]{Go}, $(1 - \gamma) \leq \del/10k <1/2k <\e/4M_1$. Hence, $1-2k(1-\gamma) >0$.


Let $f \in S(X^*)$. So, there exist $Z_1, Z_2 \in \Z$ containing $f$ and $x_1, x_2 \in S(X)$ with $\|x_1\|_{Z_1}=\|x_2\|_{Z_2}=1$ such that
\beqa
f(x_1) >\alpha, && diam(S(B(Z_1), x_1, \alpha-\del)) <\e/4M_1, \mbox{ and}\\
f(x_2) > \frac{1 + \gamma}{2}, && diam (S(B(Z_2), x_2, \gamma) <\del/10k.
\eeqa

Using Lemma ~\ref{lemsize}, we can assume that $\gamma \geq \sqrt{1 - 1/4k^2} >1 - 1/2k$.

Let $Z=Z_1 \cap Z_2$. Then $Z \in \Z$ contains $f$ and
\[
diam(S(B(Z), x_1, \alpha-\del)) <\e/4M_1 \mbox{ and } diam (S(B(Z), x_2, \gamma)) <\del/10k.
\]

Let $\eta = \gamma(1-2k(1-\gamma)) > 0$. We have,
\[
\eta = \gamma-2k\gamma + 2k\gamma^2 \geq \gamma - 2k\gamma + 2k(1 - 1/4k^2) \geq \gamma - 1/2k \geq 1-1/k.
\]
Also,
\[
\frac{f(x_2)-\eta}{f(x_2)-\gamma} \leq \frac{2(1-\eta)}{1-\gamma} \leq 2(2k+1).
\]

So, by a slight modification in \cite[Corollary 3.5]{Go},
\[
diam (S(B(Z), x_2, \eta)) \leq \frac{f(x_2)-\eta}{f(x_2)-\gamma} diam (S(B(Z), x_2, \gamma)) < \del/2.
\]

If $g \in B(Z) \setminus S(B(Z), x_1, \alpha-\del/2)$, then $\|f-g\| \geq f(x_1) - g(x_1) > \del/2$. Since $f \in S(B(Z), x_2, \eta)$, it follows that $g \notin S(B(Z), x_2, \eta)$. Thus, we obtain that
\beqa
f \in S(B(Z), x_2, \gamma) & \subseteq & S_1 = S(B(Z), x_2, \eta) \\
& \subseteq & S(B(Z), x_1, \alpha-\del/2) \subseteq S(B(Z), x_1, \alpha-\del).
\eeqa
Therefore, $diam (S_1) <\e/4M_1$.

Let $C$ be a closed convex set such that $C \subseteq B_Z [0, M]$ and $\inf f(C) \geq \e$.

Define $D = \overline{C + 3\e/4 B(X)}$. Then $D \subseteq B_Z [0, M_1]$ and $\inf f(D) \geq \e/4$. Hence, for $z \in D$ and $g \in B(Z) \cap B(f, \e/4M_1)$,
\[
g(z) \geq f(z) - \|f-g\| \|z\|_Z > \e/4 - \frac{\e M_1}{4M_1}=0.
\]
Therefore,
\[
S_1 \subseteq B(Z) \cap B(f, \e/4M_1) \subseteq \{g \in B(Z) : g(z) > 0 \mbox{ for all } z \in D\}.
\]

Since, $f \in S(B(Z), x_2, \gamma)$, $\|x_2\|_Z > \gamma >\eta$.
Thus, $(1-\eta/\gamma) \leq (1-\eta/\|x_2\|_Z)$.
Let
\[
d = \frac{d_Z(0, D)(1-\eta/\gamma)}{2} > 0.
\]
Then by Theorem ~\ref{thmbsep}, for $\lambda = M_1/2k(1-\gamma)\eta = M_1/\eta(1-\eta/\gamma) \geq M_1/\eta(1-\eta/\|x_2\|_Z)$, let $B_1=B\left[\lambda x_2, \lambda - d - \frac{3\e}{4}\right]$.

So, we have,
\[\sup_{y \in D}\|y-\lambda x_2\|_Z \leq \lambda-d
\]
\emph{i.e.},
\[
\sup_{ c \in C}\|c-\lambda x_2\|_Z \leq \lambda-d-3\e/4.
\]

Since $1/2k\eta <1/2(k-1) < \e/4M_1$, \emph{i.e.}, $M_1/2k\eta < \e/4$, we have
\beqa
\inf f (B_1) & = & f(\lambda x_2) - \lambda + d + \frac{3\e}{4} > \lambda (\gamma - 1) + d + \frac{3\e}{4}\\
& = & -\frac{M_1}{2k\eta} + d + \frac{3\e}{4} \geq d + \frac{3\e}{4}-\e/4 >\e/2.
\eeqa

Also, \[
\frac{M_1}{2k\eta(1 - \gamma)}- d - \frac{3\e}{4} \leq \frac{M_1}{2k\eta(1 - \gamma)} = K.
\]
Hence, $X$ has the AHUMIP.

$(e) \implies (c)$:
Let $\e >0$ be given. Choose $\gamma, K >0$ obtained from the hypothesis for $\e/4$ and $M=1$.
Let $\del=1-\frac{K}{K+\gamma}$.

Let $f \in S(X^*)$. So, there exists $Z \in \Z$ containing $f$ such that for any closed convex set $C \subseteq B_Z [0, M]$ with $f(C) \geq \e/4$, there exists a closed ball $B[x_0, r_0]$ with $r_0 \leq K$, $C \subseteq B_{Z}[x_0, r_0]$ and $\inf f(B[x_0, r_0]) \geq \gamma$.

Let $A_{\e}:=\{x \in B(X) : f(x) \geq \e/4\}$. Clearly, $A_{\e} \subseteq B_Z [0, 1]$ and $f(A_{\e}) \geq \e/4$.

So, there exists a closed ball $B[x_0, r_0]$ with $r_0 \leq K$, $A_{\e} \subseteq B_Z [x_0, r_0]$ and $\inf f(B[x_0, r_0]) \geq \gamma$. So, $f(x_0) \geq r_0+\gamma$.

Let $z_0=x_0/\|x_0\|_Z$. We have, $\|z_0\|_Z=1$. Then by Lemma~\ref{lemext}, there exists $z_1 \in S(X)$ such that $z_1|_Z=z_0|_Z$.

Consider $S:=S(B(Z), z_1, f(z_1)(1-\del))$.

Choose $g \in S$. So, $g \in Z$ and $g(z_1)=g(z_0) >f(z_0)(K/(K+\gamma))$.
So, $g(x_0) >f(x_0)(r_0/(r_0+\gamma)) \geq r_0$. So, $\inf g(B[x_0, r_0]) >0$.
But, $g \in B(Z)$. Thus, $g(x) >0$ for all $x \in A_{\e}$. So, $\|f-g\| <\e$.

Thus, $X^*$ is w*-AUD (I).
\end{proof}

\section{Residuality Results}
In this section, we are going to discuss residuality of AUS, uniformly smooth and UMIP norms.
Most of the techniques used in the three cases are quite similar. These are motivated by the method implemented by Chen and Lin in \cite{CL} So, we are going to provide detailed proof only for the residuality of AUS norms. Since we will talk about various norms in this section, so let's set some notations depending upon the norm under consideration.

For a subspace $Z$ of $X^*$, $B \in \B$, $C \subseteq X$, and $x \in X$, let
\bla
\item $\|x\|_{Z, B}:=\sup\{g(x) : g \in Z, \|g\|^*_B \leq 1\}$;
\item $d_{Z, B}(x, C) = \inf\{\|x-z\|_{Z, B} : z \in C\}$ and
\item $diam_{Z, B}(C) = \sup\{\|x-y\|_{Z, B} : x, y \in C\}$.
\el
We may drop the subscript in the cases when the norm or the subspace under consideration is clear.

\subsection{Residuality of AUS norms}
Let us begin with an easy observation on some quantities.
\begin{lemma} \label{l2}
Let $B_1, B_2 \in \B$. For every $\e_1 >0$ and $M_1 \geq 1$, there exist $\e_2 >0$ and $M_2 \geq 1$ such that for any $A \subseteq X$ closed convex and $Z \in \Z$, we have the following relations: \bla
\item $d_{Z, B_1}(0, A) \geq \e_1 \implies d_{Z, B_2}(0, A) \geq \e_2$.
\item $diam_{Z, B_1}(A) \leq M_1 \implies diam_{Z, B_2}(A) \leq M_2$.
\el
\end{lemma}

In particular, for $Z=X^*$, we have the following:

\begin{lemma} \label{l4}
Let $B_1, B_2 \in \B$. For every $\e_1 >0$, $M_1 \geq 1$ and $K_1 \geq 1$, there exist $\e_2 >0$, $M_2 \geq 1$ and $K_2 \geq 1$ such that for any $A \subseteq X$ closed and convex, \bla
\item $d_{B_1}(0, A) \geq \e_1 \implies d_{B_2}(0, A) \geq \e_2$
\item $diam_{B_1}(A) \leq M_1 \implies diam_{B_2}(A) \leq M_2$
\item $A \subseteq K_1 B_1 \implies A \subseteq K_2 B_2$.
\el
\end{lemma}

Following is an easy observation and we will skip its proof:
\begin{lemma} \label{lemAHUMIP}
$X$ has AHUMIP if and only if for every $\e >0$, $0 <\alpha <\e$ and $M \geq 1$, there exists $K >0$ such that for every $f \in S(X^*)$, there is $Z \in \Z$ containing $f$ such that for any closed, convex set $A$ with $A \subseteq B_Z[0, M]$ and $\inf f(A) \geq \e$, there exist $z_0 \in X$ and $r_0 > 0$ such that $r_0 \leq K$, $A \subseteq B_Z [z_0, r_0]$ and $\inf f(B[z_0, r_0]) \geq \e-\alpha$.
\end{lemma}

We now come to our main result of this section:
\begin{theorem} \label{thm2}
If $X$ has AHUMIP, then there exists a dense $G_{\del}$-subset $\B_0 \subseteq \B$ such that for every $B \in \B_0$, $(X, \|\cdot\|_B)$ has AHUMIP.
\end{theorem}

\begin{proof}
For $n, k \in \N$ and $0 <\alpha <1/2n$, consider the following object:
\beqa
\A_{\alpha, n, k} &:=& \{(A, f, Z) : A \subseteq X \mbox{ closed convex}, \, f \in S(X^*) \mbox{ and } Z \in \Z \mbox{ with} \\
&& f \in Z \mbox{ such
that there exists a closed ball } z_0 + r_0 B(X) \\
&& \mbox{satisfying }r_0 \leq k, \sup_{a \in A}\|z_0-a\|_Z \leq r_0 \mbox{ and}\\
&& \inf f(z_0 + r_0 B(X)) \geq 1/n -\alpha\}.
\eeqa
and
\beqa
\B_{n, k} &:=& \{B \in \B : \mbox{there exists } 0 <\alpha, \gamma <1/2n, K \geq 1 \mbox{ such that for any}\\
&& f \in S(X^*, \|\cdot\|^*_B), \mbox{ closed and convex } A \subseteq X \mbox{ and } Z \in \Z\\
&& \mbox{containing } f \mbox{ with }
(A, f/\|f\|, Z) \in \A_{\alpha, n, k}, \mbox{ there is a closed ball}\\
&& z_1 + r_1 B, r_1 \leq K \mbox{ satisfying } \sup_{a \in A}\|a-z_1\|_{Z, B} \leq r_1 \mbox{ and}\\
&& \inf f(z_1 + r_1 B) \geq \gamma\}.
\eeqa
Let \[
\B_0=\bigcap_{n=1}^{\infty}\bigcap_{k=1}^{\infty}\B_{n, k}.
\]

\textsc{Claim 1:} If $B \in \B_0$, then $(X, \|\cdot\|_B)$ has AHUMIP.

Let $\e >0$ and $M \geq 1$ be given. Let $h(B, B(X)) <\beta$.
Choose $M_1 \geq M(1+\beta)$ and $0 <\e_1 < \e/(1+\beta)$.

Since $(X, \|\cdot\|)$ has AHUMIP, there exist $\gamma, K >0$ that work for $\e_1, M_1$.

Let $n, k \in \N$ be such that $\gamma >1/n$, $K \leq k$.
Now, $B \in \B_0$. So, $B \in \B_{n, k}$.

Thus, there exist $0 <\alpha_1, \gamma_1 <1/2n$, $K_1 >0$ that satisfy the $\B_{n, k}$ conditions for $B$.

Let $f \in S(X^*, \|\cdot\|_B^*)$. Then, there exists $Z \in \Z$ containing $f$ that works for $f/\|f\|$ in $(X, \|\cdot\|)$ with $\e_1, M_1$ and $\gamma, K >0$.

Now, let $A$ be a closed convex set with $\sup_{a \in A}\|a\|_{Z, B} \leq M$ and $\inf f(A) \geq \e$.
So, $\sup_{a \in A}\|a\|_{Z} \leq M_1$ and $\inf (f/\|f\|)(A) \geq \e_1$.

So, there exists a closed ball $z_0 + r_0 B(X)$ satisfying $\sup_{a \in A}\|a-z_0\|_{Z} \leq r_0$, $r_0 \leq K \leq k$ and $\inf (f/\|f\|)(z_0 + r_0 B(X)) \geq \gamma >1/n-\alpha_1$.

So, $(A, f/\|f\|, Z) \in \A_{\alpha_1, n, k}$. And the $\B_{n, k}$ conditions imply that there exists a closed ball $z_1 + r_1 B$ satisfying $\sup_{a \in A}\|a-z_1\|_{Z, B} \leq r_1$, $r_1 \leq K_1$ and $\inf f(z_1 + r_1 B)\geq \gamma_1$.

So, $(X, \|\cdot\|_B)$ has AHUMIP.

\textsc{Claim 2:} For each $n, k \in \N$, $\B_{n, k}$ is open in $\B$.

Let $n, k \in \N$ and $B_0 \in \B_{n, k}$.

There exists $0 <\alpha, \gamma <1/2n$, $K >0$ satisfying the $\B_{n, k}$ conditions for $B_0$. Clearly $\alpha, \gamma <1/2$.

Let $B \in \B$ be such that $h(B, B_0) <\gamma/4K$. We will show that $B \in \B_{n, k}$.
Let $K_1=K+1$ and $\gamma_1=\frac{\gamma}{(4(1+\gamma/4K))}$.

Let $f \in S(X^*, \|\cdot\|_B^*)$, $A$ be closed convex set and $Z \in \Z$ containing $f$ such that $(A, f/\|f\|, Z) \in \A_{\alpha, n, k}$. Now, $B_0 \in \B_{n, k}$. Thus, there exists a closed ball $z_1 +r_1 B_0$ with $r_1 \leq K$, $\sup_{a \in A}\|a-z_1\|_{Z, B_0} \leq r_1$ and $\inf (f/\|f\|_{B_0})(z_1 +r_1 B_0) \geq \gamma$.
So, $f(z_1) \geq \frac{r_1+\gamma}{1+\gamma/4K}$.

Now, $z_1+r_1 B_0 \subseteq z_1+r_1(1+\gamma/4K)B$. Also, $\sup_{a \in A}\|a-z_1\|_{Z, B} \leq r_1(1+\gamma/4K)$. We have, $r_1(1+\gamma/4K) \leq K+1=K_1$.
And,
\beqa
\inf f(z_1 + r_1(1+\gamma/4K)B) & = & f(z_1)-r_1(1+\gamma/4K)\\
& \geq & \frac{r_1+\gamma}{1+\gamma/4K} - r_1(1+\gamma/4K)\\
& \geq & \frac{\gamma}{1+\gamma/4K} - K\left(1+\gamma/4K - \frac{1}{1+\gamma/4K}\right)\\
& \geq & \frac{\gamma}{(4(1+\gamma/4K))} = \gamma_1.
\eeqa
So, $B \in \B_{n, k}$.

\textsc{Claim 3:} For each $n, k \in \N$, $B_{n, k}$ is dense in $\B$.

Let $B_0 \in \B$ and $\e >0$ be given. So, we need to show that there exists $B \in \B_{n, k}$ such that $h(B, B_0) <\e$.

Let $\lambda > h(B_0, B(X))$ and let $\del = \e/2(1+\lambda)$.
Consider
\[
B = \ol{B_0 + \del B(X)}.
\]
It follows that
\[
B_0 \subseteq B = \ol{B_0+\del B(X)} \subseteq B_0+\del(1+\lambda)B_0=(1+\e/2)B_0
\]
and hence,
\[
h(B_0, B) \leq \e/2 < \e.
\]

Also, $h(B, B(X)) < \e + \lambda = \beta$ (say). So, $\frac{1}{1+\beta}\|\cdot\|_B \leq \|\cdot\| \leq (1+\beta)\|\cdot\|_B$.

We claim that $B \in \B_{n, k}$.

Let $\alpha=1/2n$, $\gamma=1/4n(1+\beta), K=k\left(1+\beta+\frac{1}{\del}\right)$.

Let $f \in S(X^*, \|\cdot\|_{B}^*)$, $A \subseteq X$ be closed convex and $Z \in \Z$ such that $(A, f/\|f\|, Z) \in \A_{\alpha, n, k}$. So, there exists a closed ball $z_0 + r_0 B(X)$ such that $r_0 \leq k$, $\sup_{a \in A}\|a-z_0\|_{Z}\leq r_0$ and $\inf (f/\|f\|)(z_0 + r_0 B(X)) \geq 1/2n$.

So, $\inf f(z_0 + r_0 B(X)) \geq 1/2n(1+\beta)$.

Let $x \in B_0$ be such that
\[
\frac{K}{\del}(f(x) + \|f\|_{B_0}) < 1/4n(1+\beta).
\]
Then,
\[
z_0 + r_0 B(X) = z_0 - \frac{r_0}{\del}x +\frac{r_0}{\del}(x + \del B(X)) \subseteq z_0 - \frac{r_0}{\del}x +\frac{r_0}{\del}B
\]
Now, $\|a-z_0+(r_0/\del)x\|_{Z, B} \leq r_0(1+\beta)+ r_0/\del \leq K$. Also,
\beqa
\inf f(z_0 - \frac{r_0}{\del}x + \frac{r_0}{\del}B) & = & \inf f(z_0 + r_0 B(X)) + \frac{r_0}{\del} \inf f(B_0) - \frac{r_0}{\del}f(x)\\
& \geq & \frac{1}{2n(1+\beta)} - \frac{r_0}{\del}\|f\|_{B_0} - \frac{r_0}{\del}f(x)\\
& \geq & \frac{1}{2n(1+\beta)} - \frac{K}{\del}(\|f\|_{B_0}+f(x))\\
& > & \frac{1}{2n(1+\beta)} - \frac{1}{4n(1+\beta)}
= \frac{1}{4n(1+\beta)}
\eeqa

So, $B \in \mathcal{B}_{n, k}$.
\end{proof}

\begin{corollary}
If $(X, \|\cdot\|)$ is AUS, then the set of equivalent norms which are AUS is residual in the set of all equivalent norms.
\end{corollary}

Now, by the arguments in the proof of \cite[Remark 4.5]{DKLR}, we have residuality of AUC norms. Hence, we obtain the following corollary which was proved in \cite{DKLR} only for reflexive spaces.

\begin{corollary}
If $X$ has an equivalent AUC norm and an equivalent AUS norm, then there is an equivalent norm which is both AUC and AUS.
\end{corollary}

In the following two subsections we are going to talk about the residuality of UMIP and uniformly smooth norms. Most of the techniques in the proof are similar to that of the AUS case. So, we will skip the proof in these cases.

\subsection{Residuality of uniformly smooth norms}
Coming to uniform smoothness, we note that \cite[Theorem 4.7]{Go} characterised it in terms of ``the hyperplane-uniform Mazur intersection property (H-UMIP)''.

\begin{definition} \rm \rm \cite{Go}
We say that $X$ has H-UMIP if for every $\e >0$ and $M \geq 1$, there is $K(\e) > 0$ such that whenever a closed convex set $C\subseteq X$ and $f \in S(X^*)$ are such that $C \subseteq B[0, M]$ and $\inf f(C) \geq \e$, there is a closed ball $B[x_0, r_0]$ in $X$ such that $C \subseteq B[x_0, r_0]$, $\inf f(B[x_0, r_0]) \geq \e/2$ and $r_0 \leq K$.
\end{definition}

Following lemma can be obtained by slight perturbation of the quantities involved in H-UMIP, hence we would skip its proof.
\begin{lemma} \label{l3}
$X$ has H-UMIP if and only if for every $\e >0$ and $M \geq 1$, there exists $0 < \gamma_0, \gamma_1 \leq \e/2$, $K >0$ such that for every closed convex set $\sup\{\|c\| : c\in C\} \leq M$ and $f \in S(X^*)$ satisfying $\inf f(C) \geq \e-\gamma_0$, there exists $z_0 \in X$ and $r_0 \leq K$ with $C\subseteq B[z_0, r_0]$ and $f(z_0)-r_0 \geq \gamma_1$.
\end{lemma}

The residuality of H-UMIP norms follows from techniques similar to that for the AHUMIP in Theorem ~\ref{thm2}. We present only the class of sets required in the proof.
 \begin{theorem} \label{thm3}
If $(X, \|\cdot\|)$ has H-UMIP, then there is a dense $G_\del$ subset $\B_0'$ of $\B$ such that $(X, \|\cdot\|_{B})$ has H-UMIP for each $B \in \B_0'$.
\end{theorem}

\begin{proof} For $n, k \in \mathbb{N}$, $k \geq 2$, let
\beqa
\mathcal{B}_{n, k}' &:=& \{B \in \mathcal{B} : \mbox{ there exist } 0 < \gamma_0, \gamma_1 < 1/2n \mbox{ and } K > 0 \mbox{ such that for}\\
&& \mbox{every } A \mbox{ closed, convex with } A \subseteq B[0, k], \mbox{ and } f \in S(X^*, \|\cdot\|^*_B)\\
&& \mbox{with } \inf f(A) \geq 1/n-\gamma_0, \mbox{we have } z_0\in X, \, r\leq K, \mbox{ so that}\\
&& A \subseteq z_0 +rB, \mbox{ and } \inf f(z_0 +rB) >\gamma_1\}.
\eeqa

As in Theorem ~\ref{thm2}, we can show that
\[
\B_0'=\bigcap_{n=1}^{\infty}\bigcap_{k=1}^{\infty}\B_{n, k}'
\]
is a $G_{\delta}$ dense subset of $\B$ and each element of $\B_0'$ has H-UMIP.
\end{proof}

\begin{corollary} \label{cor}
If $(X, \|\cdot\|)$ is uniformly smooth, then the set of equivalent uniformly smooth norms is residual in the set of all equivalent norms on $X$.
\end{corollary}

Notice that if $B \in \B$ then the corresponding dual ball $B^* = B^\circ$, the polar of $B$. It follows that if $B_1, B_2 \in \B$, then $h(B^*_1, B^*_2) = h(B_1, B_2)$.

Since in the dual of a reflexive space, any equivalent norm on $X^*$ is a dual norm, we have the following observation:
\begin{remark} \label{rmk} \rm
If $\B_0$ is residual in $\B$ on a reflexive space, then the corresponding set $\B_0^*$ of dual norms is residual in the set of equivalent norm on $X^*$, and vice-versa.
\end{remark}

\begin{corollary}
If $(X, \|\cdot\|)$ is uniformly convex, then the set of equivalent uniformly convex norms is residual in the set of all equivalent norms on $X$.
\end{corollary}

\begin{corollary}
If $X$ is super-reflexive, then there exists a norm on $X$ which is both uniformly convex and uniformly smooth.
\end{corollary}

\subsection{Residuality of norms with UMIP}

\begin{definition} \label{UMIP} \rm \cite{WZ}
We say that a Banach space $X$ has the \emph{Uniform Mazur Intersection Property (UMIP)} if for every $\e >0$ and $M(\e) \geq 2$, there is $K(\e) > 0$ such that whenever a closed convex set $C\subseteq X$ and a point $p\in X$ are such that $diam(C) \leq M(\e)$ and $d(p, C) \geq \e$, there is a closed ball $B \subseteq X$ of radius $\leq K(\e)$ such that $C \subseteq B$ and $d(p, B) \geq \e/2$.
\end{definition}

The proof of following lemma is an easy adaptation of the proof of \cite[Theorem 2.11]{BGG}. We skip the details.

\begin{lemma}
$X$ has UMIP if and only if for every $\e >0$ and $M \geq 2$, there exist $K >0$ and $0 < \delta \leq \e$ such that for every closed, bounded, convex $A \subseteq X$ with $d(0, A) \geq \e$ and $diam(A) \leq M$, we have a closed ball $B[z_0, r]$ with $r \leq K$ such that $A \subseteq B[z_0, r]$ and $d(0, B[z_0, r]) \geq \delta$.
\end{lemma}

\begin{theorem} \label{thm4}
If $(X, \|\cdot\|)$ has UMIP, then there is a dense $G_\del$ subset $\B_0$ of $\B$ such that $(X, \|\cdot\|_B)$ has UMIP for each $B \in \B_0$.
\end{theorem}

\begin{proof}
For $n, k, \in \mathbb{N}$, $k \geq 2$ let
\beqa
\mathcal{A}_{n, k} & := & \{A \subseteq X : A \mbox{ is closed, convex and bounded}, \, d(0, A) \geq 1/n, \\
&& diam(A) \leq k \} \mbox{ and }\\
\mathcal{B}_{n, k} & := & \{B \in \mathcal{B} : \mbox{ there exist } 0< \gamma < 1/n, \, K > 0 \mbox{ such that for every}\\
&& A\in \mathcal{A}_{n, k}, \mbox{ we have } z_0\in X, \, r\leq K, \mbox{ so that } A \subseteq z_0 +rB, \\
&& d_B(0, z_0 +rB) >\gamma\}.
\eeqa

Then, again as in Theorem ~\ref{thm2}, we can show that
\[
\mathcal{B}_0 = \bigcap_{n=1}^{\infty} \bigcap_{k=1}^{\infty} \mathcal{B}_{n, k}
\]
is a dense $G_\del$ set in $\mathcal{B}$ and each element of $\B_0$ has UMIP.
\end{proof}

\begin{remark} \rm
The proof here is not exactly a simple adaptation of the proof for AHUMIP case. But it still is not difficult and can be omitted.
\end{remark}

\section{Two open problems}

\subsection{Residuality of Fr\'echet smooth norms}

Following theorem characterises smoothness and Fr\'echet smoothness in terms of extreme and w*-denting points of $B(X^*)$ respectively.

\begin{theorem} \label{TF1}
\bla
\item $X$ is smooth if and only if every $f \in NA$ is an extreme point of $B(X^*)$.
\item $X$ is Fr\'echet smooth if and only if every $f \in NA$ is a w*-denting point of $B(X^*)$.
\el
\end{theorem}

\begin{proof}
$(a)$: Let $X$ be smooth and $f \in NA$. Let $x \in S(X)$ be such that $f \in D(x)$. Since $D(x)$ is a singleton face of $B(X^*)$, $f$ is an extreme point.

Conversely, suppose every element of $NA$ is extreme. Let $x \in S(X)$ and $g, h \in D(x)$. Then, $\frac{(h+g)(x)}{2}=1$. Thus, $\frac{f+g}{2} \in D(x) \subseteq NA$. Therefore, $\frac{f+g}{2}$ is an extreme point. Hence, $f=g = h$. So, $X$ is smooth.

$(b)$: Suppose $X$ is Fr\'echet smooth. Let $f \in NA$. Let $x \in S(X)$ be such that $f \in D(x)$. Since $x$ is a point of Fr\'echet differentiability, $f$ is w*-strongly exposed point, and hence, a w*-denting point of $B(X^*)$.

Conversely, let all $f \in NA$ be w*-denting. So, every $f \in NA$ is extreme. Thus, by part $(a)$, $X$ is smooth. Hence, the duality map $D$ is single-valued and norm-w* continuous. But, each $f \in NA$ is w*-denting, and hence, w*-PC. Hence, the duality map is norm-norm continuous. So $X$ is Fr\'echet smooth.
\end{proof}

Now, we obtain a ball separation characterisation of Fr\'echet smoothness using the characterisation of w*-denting points in \cite[Corollary 1.6]{CL} and Theorem ~\ref{TF1}.

\begin{theorem}
$X$ is Fr\'echet smooth if and only if for every closed bounded set $C$ and $f \in NA$ satisfying $\inf f(C) >0$, we have $z_0 \in X$ and $r_0 >0$ such that $C \subseteq z_0 +r_0 B(X)$ and $\inf f(z_0 + r_0 B(X)) >0$.
\end{theorem}

Residuality of Fr\'echet smooth norms seems to be open. Now, considering the connection between residuality and ball separation properties, we can ask the following:

\begin{question}
If $X$ has a Fr\'echet smooth norm, then is the set of equivalent norms which are Fr\'echet smooth residual in $\B$?
\end{question}

If the answer to the above question is affirmative, then using the residuality of LUR norms, it's easy to answer the following open problem:
\begin{question}
If $X$ has an equivalent Fr\'echet smooth norm and an equivalent LUR norm, then does there exist an equivalent norm which is both Fr\'echet smooth and LUR?
\end{question}

\subsection{Uniformly convex renorming and UMIP}
Another long-standing open question is:
\begin{question} \label{q3}
Does UMIP imply super-reflexivity?
\end{question}
This has been posed as an open question in \cite{GMZ}. Now, consider the following isometric problem:
\begin{question} \label{q4}
If $X$ has both H-MIP and UMIP, does $X$ have H-UMIP?
\end{question}
We observe that if the answer to Question~\ref{q4} is affirmative, then Question~\ref{q3} also has a positive answer in separable Banach spaces.

It is well known that separable MIP spaces have separable dual \cite[Theorem 2.1]{GGS}, and hence, is Asplund. It follows that $X^*$ has a dual LUR renorming. This, in turn, implies that $X$ has H-MIP. It follows that a separable MIP space also has an equivalent H-MIP renorming.

Now, from \cite[Corollary 5.7]{CL} it follows that H-MIP is residual. And Theorem~\ref{thm4} shows that so is UMIP. Thus, any separable UMIP space has an equivalent norm with both UMIP and H-MIP. Therefore, if the answer to Question~\ref{q4} is affirmative, there is an equivalent norm on $X$ that has H-UMIP, and hence, is uniformly smooth. That is, Question~\ref{q3} also has a positive answer for separable spaces.

\end{document}